\documentclass[reqno,a4paper]{amsart}

\usepackage[dvipsnames,usenames]{color}
\usepackage{amssymb,amsfonts,amscd,amsthm,enumerate,pb-diagram}
\usepackage{amsbsy,bm,amsmath}
\usepackage[normalem]{ulem}
\usepackage{tikz-cd}
\usepackage[shortlabels]{enumitem}
\usepackage[hidelinks]{hyperref}
\usepackage[shortcuts]{extdash}

\newtheorem{theoremalpha}{Theorem} 

\newtheorem{lemma}{Lemma}[section]

\theoremstyle{definition}
\newtheorem{definition}{Definition}[section]

\newtheorem{example}{Example}[section]

\numberwithin{equation}{section}

\newcommand{\R}{\mathbb{R}}
\newcommand{\C}{\mathbb{C}}
\newcommand{\Sl}{\mathrm{SL}(3,\R)}
\newcommand{\id}{\mathrm{Id}}

\title[Almost complex totally geodesic surfaces in $\frac{\text{SL}(3,\R)}{\R\times \text{SO}(2)}$]
{Almost complex totally geodesic surfaces\\ in the nearly Kähler $\frac{\text{SL}(3,\R)}{\R\times \text{SO}(2)}$}

\author[M. Anarella]{Mateo Anarella}
\address{M.\ Anarella, Department of Mathematics, Faculty of Science, Hokkaido University, Kita~10, Nishi~8, Kita\-/Ku, Sapporo, Hokkaido, 060-0810, Japan}
\email{mateo.anarella@math.sci.hokudai.ac.jp}

\author[X. X. Cheng]{Xiuxiu Cheng}
\address{X.\ X.\ Cheng, School of Mathematics and Statistics, Zhengzhou University, Zhengzhou~450001, People’s Republic of China}
\email{xxc19@zzu.edu.cn}

\author[M. D'haene]{Marie D'haene}
\address{M.\ D'haene, Department of Mathematics, KU Leuven, Celestijnenlaan 200B, Box 2400, 3001 Leuven, Belgium}
\email{marie.dhaene@kuleuven.be}

\author[Z. J. Hu]{Zejun Hu}
\address{Z.\ J.\ Hu, School of Mathematics and Statistics, Zhengzhou University, Zhengzhou~450001, People's Republic of China}
\email{huzj@zzu.edu.cn}

\author[L. Vrancken]{Luc Vrancken}
\address{L.\ Vrancken, Ceramaths, Université Polytechnique Hauts\-/de\-/France, F-59313 Valenciennes, France;
Ceramaths, INSA Hauts de France, F-59313 Valenciennes, France; Department of Mathematics, KU
Leuven, Celestijnenlaan 200B, Box 2400, 3001 Leuven, Belgium}
\email{luc.vrancken@uphf.fr}

\thanks{M.\ Anarella was partially supported by FWO and FNRS under EOS project G0I2222N.
X.\ X.\ Cheng was supported by the NSF of Henan Province (grant no.\ 252300421476) and the NSF of China
(grant no.\ 12001494).
M.\ D'haene was supported by Methusalem grant METH/21/03-long term structural funding of the Flemish Government and FWO (Research Foundation Flanders) grant K217724N.
Z.\ J.\ Hu was supported by NSF of China, grant no.\ 12171437.}
\date{}

\begin{document}

\begin{abstract}
We give a detailed description of the nearly K\"{a}hler $\frac{\Sl}{\R\times \mathrm{SO}(2)}$, which is one of the pseudo\-/Riemannian counterparts of the flag manifold~$F(\mathbb{C}^3)$.
The main result is the classification of totally geodesic almost complex surfaces in this space.
\end{abstract}

\maketitle

\thispagestyle{empty}

\noindent
Let $(\bar{M},g)$ be a (pseudo-)Riemannian manifold equipped with an almost complex structure $J$ which is compatible with~$g$.
We call $(\bar{M},g,J)$ a nearly K\"{a}hler manifold if $\bar\nabla J$ is skew-symmetric,
where $\bar{\nabla}$ is the Levi-Civita connection associated to~$g$.
If, moreover $\bar\nabla J$ is non-degenerate, we say that $\bar M$ is \emph{strict} nearly Kähler.
In 1970, Gray~\cite{G3} generalized many well known theorems and formulas about the topology and geometry of K\"{a}hler manifolds to nearly K\"{a}hler manifolds.
Since then, the study of nearly K\"{a}hler manifolds attracted a lot of attention~\cite{G4,G3,G2,G1,S,K,N}.

Two- and four-dimensional nearly K\"{a}hler manifolds are always K\"{a}hler. 
Therefore, the study of strict nearly K\"{a}hler manifolds starts from dimension~six. 
Butruille~\cite{B} proved that, in the Riemannian case, the only complete simply connected homogeneous strict nearly K\"{a}hler manifolds in dimension six are $\mathbb{S}^6$, $\mathbb{S}^3\times\mathbb{S}^3$, $\mathbb{C}P^3$ and $F(\mathbb{C}^3)$.
However, such a classification is still open for the pseudo\-/Riemannian case.

Kath~\cite{K} provided four pseudo\-/Riemannian analogues of the flag manifold $F(\C^3)$: 
\begin{equation*}
      \begin{tikzcd}[ampersand replacement=\&, column sep=small,row sep=small]
            F(\C^3) \arrow[r]\arrow[dr, to path=|- (\tikztotarget)]
            \arrow[ddr, to path=|- (\tikztotarget)] \arrow[dddr, to path=|- (\tikztotarget)]\& \frac{\mathrm{SU}(2,1)}{\mathrm{U}(1) \times \mathrm{U}(1)} \& \text{nearly Kähler,}\\      
            \& \frac{\mathrm{SU}(2,1)}{\mathrm{SO}(1,1)\times \mathrm{U}(1)} \& \text{nearly para-Kähler,}\\ 
            \& \frac{\Sl}{\R\times \mathrm{SO}(2)} \& \text{nearly Kähler,}\\      
            \& \frac{\Sl}{(\R^+\times \R^+)\cup (\R^-\times \R^-)} \& \text{nearly para-Kähler.}\\   
      \end{tikzcd}
\end{equation*}
Note that a nearly para-Kähler manifold is a (pseudo-)Riemannian manifold $(M,g)$ equipped with a $(1,1)$-tensor $J$ such that $J^2=\id$, $g(JX,JY)=-g(X,Y)$ and $\bar\nabla J$ is skew-symmetric, where $\bar\nabla$ is the Levi-Civita connection associated to $g$.

Cwiklinski and Vrancken~\cite{CV} studied the nearly K\"{a}hler structure on the flag manifold $F(\mathbb{C}^3)$ and its pseudo\-/Riemannian counterpart $\frac{\mathrm{SU}(2,1)}{\mathrm{U}(1)\times \mathrm{U}(1)}$, and they classified the totally geodesic almost complex submanifolds of these two spaces.

In the current article, we focus on another pseudo\-/Riemannian dual of the flag manifold, namely, $\frac{\Sl}{\R\times \mathrm{SO}(2)}$.
We describe the structure of this space in Section~\ref{sec:nk-structure}.
Then, we classify totally geodesic almost complex surfaces in Section~\ref{sec:tot-geod-ac-surfaces}.
Here, a submanifold $\Sigma$ is called almost complex if $J$ preserves the tangent space, i.e. $J(T\Sigma) \subseteq T\Sigma$. 

Before stating the main theorem, we introduce the following notations.
Let $\mathbb{R}^{n+1}_\nu$ be the semi-Euclidean space with metric
\[
\langle v,w\rangle_p=-\sum_{i=1}^\nu v_i(p) w_i(p) +\sum_{i=\nu+1}^{n+1} v_i(p) w_i(p) , \quad \text{ where }
\ v=\sum_{i=1}^{n+1} v_i \partial_i,\  w=\sum_{i=1}^{n+1} w_i\partial_i,
\]
where $\{\partial_i\}_{i=1}^{n+1}$ is a coordinate frame of $\mathbb{R}^{n+1}_\nu$.
We denote the pseudosphere of radius $r>0$ with index $\nu$ in $\mathbb{R}^{n+1}_\nu$ by
$$S^n_\nu(r)=\{p\in \mathbb{R}^{n+1}_\nu\mid\langle p,p\rangle=r^2\}.$$

\begin{theoremalpha}
    \label{thm:main}
    Let $f:\Sigma\rightarrow \frac{\Sl}{\mathbb{R}\times \mathrm{SO}(2)}$ be a totally geodesic
    almost complex surface.
    Then,
    \begin{enumerate}[(1)]
    \item $\Sigma$ is locally isometric to $S^2_2(\frac{1}{2})$ and $f$ is locally 
    congruent to the immersion in Example~\ref{ex:e1+e2}, an orbit through $[\id]$ of a subgroup $\mathrm{SL}(2,\R)\subset\Sl$, or \label{theoA1}
    \vspace*{1ex}
    
    \item $\Sigma$ is locally isometric to $S^2(1)$ and $f$ is locally 
    congruent to the immersion in Example~\ref{ex:e3+e5}, an orbit through $[\id]$ of a subgroup $\mathrm{SO}(3)\subset\Sl$, or
    \vspace*{1ex}
    
    \item $\Sigma$ is locally isometric to $S^2_2(1)$ and $f$ is locally 
    congruent to the immersion in Example~\ref{ex:e3-e5}, an orbit through $[\id]$ of a subgroup $\mathrm{SO}^+(2,1)\subset\Sl$, or
    \vspace*{1ex}
    
    \item $\Sigma$ is locally isometric to $\mathbb{R}^2_2$ and $f$ is locally 
    congruent to the immersion in Example~\ref{ex:e1+e3+e5}, an orbit through $[\id]$ of a subgroup $\R^2\subset\Sl$, or
    \vspace*{1ex}
    
    \item $\Sigma$ is a degenerate surface and $f$ is locally congruent to the immersion in Example~\ref{ex:degenerateexample}, an orbit through $[\id]$ of a subgroup $\R^2\subset\Sl$.
    \end{enumerate}
\end{theoremalpha}

\section{Preliminaries}
\noindent
\begin{definition}
Let $\Sigma$ be a submanifold of an almost Hermitian manifold $(M,g,J)$. 
Then $\Sigma$ is said to be almost complex if $J$ preserves the tangent space, i.e.\ $J(T\Sigma)\subseteq T\Sigma$.
\end{definition}
\begin{definition}
Let $\Sigma$ be a (degenerate) submanifold of a pseudo\-/Riemannian space $M$ with Levi-Civita connection $\tilde{\nabla}$. 
We say that $\Sigma$ is totally geodesic if for every $p\in\Sigma$ and $v\in T_pM$ the geodesic with initial velocity $v$ lies initially in $\Sigma$.
\end{definition}

\begin{lemma}
    The following are equivalent:
    \begin{enumerate}
        \item $\Sigma$ is totally geodesic. \label{totgeogeodesics}
        \item For $X,Y\in\mathfrak{X}(\Sigma)$, $\tilde{\nabla}_XY$ is tangent to $\Sigma$. \label{totgeoconnection}
    \end{enumerate}
    If moreover, $\Sigma$ is non-degenerate, then the assertions above are equivalent to the following:
    \begin{enumerate}\addtocounter{enumi}{2}
        \item The geodesics of $\Sigma$ are also geodesics of $M$.
        \item The second fundamental form of $\Sigma$ vanishes everywhere.
    \end{enumerate}
\end{lemma}
\begin{proof}
    If the submanifold is non-degenerate, then the statement reduces to Proposition~4.13 in~\cite{oneill}.
    Hence, here we only prove the equivalence between~(\ref{totgeogeodesics}) and~(\ref{totgeoconnection}) when $\Sigma$ is degenerate.
    
    (2) $\Rightarrow$ (1). Suppose that $\tilde{\nabla}$ preserves the tangent space of $\Sigma$.
    Then, $\tilde{\nabla}$ induces an affine connection on $\Sigma$, with the property that for every tangent vector $v\in T_p\Sigma$ there exists a unique geodesic $\gamma$ in $\Sigma$ with initial velocity $v$ starting at $p$.
    By uniqueness,  the geodesics obtained by viewing $\tilde\nabla$ as an affine connection on $M$ and on $\Sigma$, have to coincide.
    
    (1) $\Rightarrow$ (2). Suppose that $\Sigma$ is totally geodesic.
    Fix a $p \in \Sigma$ and in a neighbourhood $U \subset \Sigma$ of $p$, choose a smooth complement $\mathcal{N}$ of $TU$ in $TM$. 
    We write $(\tilde{\nabla}_XY)_p=\tau_p(X,Y)+\nu_p(X,Y)$ , where $\tau$ has values in $TU$ and $\nu$ in $\mathcal{N}$.
    As for the non-degenerate case, $\tau$ is an affine connection on $U$ and $\nu$ is a bilinear symmetric tensor. 
    Take any $v\in T_pU$ and let $\gamma_v$ be the geodesic of $M$ with initial velocity $v$.
    Since $\gamma$ stays initially in $\Sigma$, we can decompose $\tau_p(v,\gamma_v')+\nu_p(v,v)=\tilde{\nabla}_{v}\gamma_v'=0$. 
    Hence $\nu_p(v,v)=0$ for all $v\in T_pU$. 
    Since $\nu$ is symmetric, it follows that $\nu=0$.
\end{proof}

\section{The nearly Kähler structure on \texorpdfstring{$\frac{\Sl}{\R\times \mathrm{SO}(2)}$}{SL(2,R)/RxSO(2)}}
\label{sec:nk-structure}
\noindent
Recall that $\Sl=\{a\in \mathrm{GL}(3,\mathbb{R})\mid{\rm det}\, a=1\}$ with associated Lie algebra $\mathfrak{s}\mathfrak{l}(3,\mathbb{R}) =\{A \in \mathbb R^{3\times 3}\mid {\rm Tr}\, A=0\}$.
We denote by $H \cong \mathbb{R}\times \mathrm{SO}(2)$ the Lie subgroup of $\Sl$ consisting of elements
\begin{equation}
\label{eq:element-isotropy}
    \begin{pmatrix}
    e^t\cos s&e^t\sin s&0\\
    -e^t\sin s&e^t\cos s&0\\
    0&0&e^{-2t}\\
    \end{pmatrix},
\end{equation}
for $t \in \R$ and $s \in [0,2\pi)$.

A basis of the Lie algebra $\mathfrak{sl}(3,\mathbb{R})$ is given by the following matrices:
\begin{equation*}
\begin{alignedat}{3}
e_1 &= \begin{pmatrix}
    1&0&0\\
    0&-1&0\\
    0&0&0
\end{pmatrix},\quad
&& e_2 = \begin{pmatrix}
    0&1&0\\
    1&0&0\\
    0&0&0
\end{pmatrix},\quad
&& e_3 = \begin{pmatrix}
    0&0&\sqrt{2}\\
    0&0&0\\
    0&0&0
\end{pmatrix}, \\[1ex]
e_4 &= \begin{pmatrix}
    0&0&0\\
    0&0&\sqrt{2}\\
    0&0&0
\end{pmatrix}, \quad
&& e_5 = \begin{pmatrix}
    0&0&0\\
    0&0&0\\
    -\sqrt{2}&0&0
\end{pmatrix},\quad
&& e_6 = \begin{pmatrix}
    0&0&0\\
    0&0&0\\
    0&-\sqrt{2}&0
\end{pmatrix}, \\[1ex]
e_7 &= \begin{pmatrix}
    \frac{1}{\sqrt{3}}&0&0\\
    0&\frac{1}{\sqrt{3}}&0\\
    0&0&\frac{-2}{\sqrt{3}}
\end{pmatrix},\quad
&& e_8 = \begin{pmatrix}
    0&1&0\\
    -1&0&0\\
    0&0&0
\end{pmatrix}.
\end{alignedat}
\end{equation*}
We denote by $\mathfrak{h}$, resp. $\mathfrak{m}_1$, $\mathfrak{m}_2$, $\mathfrak{m}_3$, the vector spaces spanned by $\{e_7,e_8\}$, resp.$\{e_1,e_2\}$, $\{e_3,e_4\}$, $\{e_5,e_6\},$
Note that $\mathfrak{h}$ is the Lie algebra of $H$.
Since $\mathfrak m = \mathfrak m_1 \oplus \mathfrak m_2 \oplus \mathfrak m_3$ is $\mathrm{Ad}(H)$-invariant and $\mathfrak{sl}(3,\R) = \mathfrak h \oplus \mathfrak m$, the homogeneous space $\Sl/H$ is reductive.
More specifically, the adjoint action of $H$ on $\mathfrak m$ is as follows: 
\[
\begin{aligned}
\mathrm{Ad}(h)(e_1,e_2)
&=(e_1,e_2)\left( \begin{matrix}
\cos 2s&\sin 2s\\
-\sin 2s&\cos 2s\\
\end{matrix}
\right), \\
\mathrm{Ad}(h)(e_3,e_4)
&=(e_3,e_4)\left( \begin{matrix}
\cos s&\sin s\\
-\sin s&\cos s\\
\end{matrix}
\right)e^{3t}, \\
\mathrm{Ad}(h)(e_5,e_6)
&=(e_5,e_6)\left( \begin{matrix}
\cos s&\sin s\\
-\sin s&\cos s\\
\end{matrix}
\right)e^{-3t},
\end{aligned}
\]
where $h\in H$ is of the form~\eqref{eq:element-isotropy}.
Note that each $\mathfrak m_i$ is $\mathrm{Ad}(H)$-invariant.
Therefore, 
by left extending $\mathfrak m_i$ ($i=1,2,3$) we may define three rank-two distributions $V_1$, $V_2$ and $V_3$ on $\Sl/H$ which are $\Sl$-invariant.

We now define a bi-invariant metric on $\Sl$ by $\langle dL_aX,dL_aY\rangle = -\frac{1}{2} \mathrm{Tr} (X Y)$,
for $X,Y\in\mathfrak{sl}(3,\R)$.
At the identity, the non-zero components of this metric with respect to $\{e_1,\ldots,e_8\}$ are given by
\begin{equation*}
\langle e_1,e_1 \rangle =\langle e_2,e_2 \rangle = \langle e_7,e_7 \rangle = -1, \quad
\langle e_3,e_5 \rangle = \langle e_4,e_6 \rangle = \langle e_8,e_8 \rangle = 1.
\end{equation*}
The restriction to $V_1\oplus V_2\oplus V_3$ of $\langle \cdot,\cdot \rangle$ is an indefinite metric with signature 4, which is $\mathrm{Ad}(H)$-invariant.
Therefore we may define an $\Sl$-invariant metric $g$ on the quotient $\Sl/H$.
Moreover, note that 
\[ 
\langle e_i,[e_j,e_k]\rangle=\langle [e_i,e_j],e_k\rangle,\quad (i,j,k \in \{1,\ldots,6\})
\]
and hence $(\Sl/H, g)$ is \emph{naturally} reductive homogenous space.

We may define two linear complex structures, $J$ and $J_1$, on $\mathfrak{m}$ determined by
\begin{equation}
\label{eqn:ji-defs}
    \begin{aligned}
    Je_1&=-e_2, \qquad  && Je_3=e_4, \qquad && Je_5=e_6\\
    J_1e_1&=e_2, \qquad && J_1e_3=e_4,  \qquad && J_1e_5=e_6.\\
    \end{aligned}
\end{equation}
A direct computation shows that $J$ and $J_1$ are compatible with $\langle\cdot,\cdot\rangle$ and commute with the adjoint action of $H$ on $\mathfrak m$.
By left extending these structures, we obtain two $\Sl$-invariant almost complex structures on $\Sl/H$ which are compatible with~$g$.
In fact, $J_1$ is an integrable complex structure and $J$ turns $(\Sl/H, g, J)$ into a nearly Kähler manifold. 

In order to obtain an expression of the curvature tensor of $(\Sl/H,g)$,
we define a linear operator $F$, known as an $F$-structure~\cite{yano-kon}, on $\mathfrak{m}$ as follows:
\begin{equation}
\label{eqn:F-def}
    Fe_1 = Fe_2=0,\quad Fe_3=e_4,\quad Fe_4=-e_3,\quad Fe_5=-e_6,\quad Fe_6=e_5.
\end{equation}
A direct computation shows that $F$ is symmetric with respect to $\langle\cdot,\cdot\rangle$, commutes with the adjoint action of $H$ on $\mathfrak m$, and $F^3+F = 0$.
Consequently, by left extension, we may define an $\Sl$-invariant (1,1)-tensor $F$ on $\Sl$.  
Additionally, $J$, $J_1$ and $F$ pairwise commute.

Using the previously introduced structures, we may express the curvature tensor of $(\Sl/H, g)$ as follows:
\begin{equation}
\label{eqn:curvature-tensor}
\begin{aligned}
\bar{R}(X,Y)Z =&
    \ \frac{5}{2}\left( g(Y,Z)X-g(X,Z)Y)\right)\\
    & -\frac{3}{4}\left( g(JY,Z)JX-g(JX,Z)JY+2g(X,JY)JZ\right)\\
    & +\frac{9}{4}\left( g(Y, J_1JZ)J_1JX-
      g(X,J_1JZ)J_1JY\right)\\
    & +\frac{3}{4}\left(g(J_1JY, Z)X-
      g(J_1JX,Z)Y+g(Y, Z)J_1JX-
      g(X,Z)J_1JY)\right)\\
    & -3\left(g(Y, FZ)FX-g(X, FZ)FY\right).
\end{aligned}
\end{equation}

\section{Totally geodesic almost complex surfaces of \texorpdfstring{$\frac{\Sl}{\R\times\mathrm{SO}(2)}$}{SL(2,R)/RxSO(2)}}
\label{sec:tot-geod-ac-surfaces}

\noindent
Let $\Sigma$ be an almost complex surface of $(\Sl/H,g,J)$, where $J$ is defined by~\eqref{eqn:ji-defs}.
Note that, if $\Sigma$ is non-degenerate, the induced metric is either positive or negative definite since $g(JX,JX)= g(X,X)$ and $g(X,JX)=0$ for any $X\in \mathfrak X(\Sigma)$.

\subsection{Examples of totally geodesic almost complex surfaces}

To obtain our classification result, we introduce five surfaces in
$\frac{\Sl}{\R\times\mathrm{SO}(2)}$ and prove that they are totally geodesic and almost complex.
In order to do so, we use the two lemmas below.

For non-degenerate surfaces, we use a generalization of~\cite[Lemma 2.1]{alberto-nearly-kahler} to pseudo\-/Riemannian ambient spaces, and for degenerate surfaces we use a result analogous to~\cite[Theorem 4.3]{alberto-nearly-kahler}.
\begin{lemma}
\label{lemma:sff}
    Let $G/H$ be a naturally reductive pseudo\-/Riemannian homogeneous space with reductive decomposition $\mathfrak g = \mathfrak h \oplus \mathfrak m$.
    Let $K \subset G$ be a Lie subgroup such that the orbit $K\cdot o$ is a non-degenerate submanifold, where $o=[\mathrm{Id}]$. 
    Denote by $\mathfrak k_{\mathfrak m}$ and $\mathfrak k_{\mathfrak m}^\perp$, respectively, the tangent and normal space at $o$ of $K\cdot o$ (viewed as subspaces of~$\mathfrak m$).
    Then, the second fundamental form at $o$ of $K\cdot o$ is given by 
    \[
    h(X, Y) = \frac12 [X,Y]_{\mathfrak k_{\mathfrak m}^\perp},
    \]
    where $X,Y \in \mathfrak k_{\mathfrak m}$.
\end{lemma}
A subalgebra $\mathfrak{k}$ of $\mathfrak{g}$ is said to be canonically embedded in $\mathfrak{g}$ if  $\mathfrak{k}=(\mathfrak{k}\cap\mathfrak{h})\oplus(\mathfrak{k}\cap\mathfrak{m})$.
\begin{lemma}\label{lemma:totgeocanembedded}
    Let $\Sigma$ be a degenerate submanifold of a naturally reductive pseudo\-/Riemannian manifold $G/H$, such that $\Sigma=K\cdot [\id]$, for a Lie subgroup $K$ with canonically embedded Lie algebra $\mathfrak{k}$ in $\mathfrak{g}$. 
    Then $\Sigma$ is a totally geodesic submanifold.
\end{lemma}

\begin{example}
    \label{ex:e1+e2}
For $u \in \R$ and $v \in [0,2\pi)$, consider
\begin{equation*}
    \begin{aligned}
    f_1(u,v)
    &=\pi\circ\exp (u(\cos v\, e_1+\sin v\, e_2)) \\
    &=\left[\begin{pmatrix}
    \cosh u+\cos v\sinh u&\sin v \sinh u&0\\
    \sin v \sinh u&\cosh u-\cos v\sinh u&0\\
    0&0&1\\
    \end{pmatrix}\right]. 
    \end{aligned}
\end{equation*}
The tangent space of the image of $f_1$ is $V_1$, and hence is preserved by $J$.
As such, this surface is almost complex.
Additionally, it is a negative definite surface with constant sectional curvature $4$ and it is the orbit through $[\id]$ of the subgroup $\mathrm{SL}(2,\R)\subset \Sl$.
As a consequence, the image of $f_1$ is isometric to $S^2_2(\frac 12)$.
Finally, since $[e_1,e_2]$ is contained in $\mathfrak h$, it follows from Lemma~\ref{lemma:sff} that this surface is totally geodesic.

\end{example}

\begin{example}
\label{ex:e3+e5}
Let $X = \frac{1}{\sqrt{2}}(e_3+e_5)$.
For $u \in \R$ and $v \in [0,2\pi)$, consider
\begin{equation*}
    \begin{aligned}
    f_2(u,v) 
    &= \pi \circ \exp\left( 2u (\cos v\, X + \sin v \, JX)\right) \\
    &= \left[
    \begin{pmatrix}
     \cos ^2 u-\cos 2 v \sin ^2 u & -\sin ^2 u \sin 2 v & \cos  v \sin 2 u \\
     -\sin ^2 u \sin 2 v & \cos ^2 u+\cos 2 v \sin ^2 u & \sin 2 u \sin  v \\
     -\cos  v \sin 2 u & -\sin 2 u \sin v & \cos 2 u 
    \end{pmatrix} \right].
    \end{aligned}
\end{equation*}
The image of $f_2$ is an almost complex surface with constant sectional curvature~$1$. 
Moreover, it is the orbit through $[\id]$ of the Lie subgroup $\mathrm{SO}(3)\subset\Sl$. 
As a consequence, the image of $f_2$ is isometric to a sphere $S^2(1)$.
Finally, since $[X,JX]$ is contained in $\mathfrak h$, it follows from Lemma~\ref{lemma:sff} that this surface is totally geodesic.

\end{example}

\begin{example}
\label{ex:e3-e5}
Let $X=\frac{1}{\sqrt{2}}(e_3-e_5)$. 
For $u \in \R$ and $v \in [0,2\pi)$, consider
\begin{equation*}
    \begin{aligned}
    f_3(u,v) 
    &= \pi\circ\exp (2u(\cos v \, X+\sin v \, JX)) \\
    &= \left[
    \begin{pmatrix}
     \cosh ^2u+\cos 2 v \sinh ^2u & \sin 2 v \sinh ^2u & \cos v\sinh 2 u \\
     \sin 2 v \sinh ^2u & \cosh ^2u-\cos 2 v \sinh ^2u & \sin v \sinh 2 u \\
     \cos v \sinh 2 u & \sin v \sinh 2 u & \cosh 2 u 
    \end{pmatrix} \right].
    \end{aligned}
\end{equation*}
The image of $f_3$ is a negative definite almost complex surface with constant sectional curvature $1$. 
Moreover, it is the orbit through $[\id]$ of the Lie subgroup $\mathrm{SO}^+(2,1)\subset\Sl$ through the point $[\id]$.
As a consequence, this surface is isometric to $S^2_2(1)$.
Finally, since $[X,JX]$ is contained in $\mathfrak h$, it follows from Lemma~\ref{lemma:sff} that this surface is totally geodesic.

\end{example}

\begin{example}
\label{ex:e1+e3+e5}
Let $X=\tfrac{1}{\sqrt{3}}e_1+e_3-\tfrac{1}{3}e_5$. 
For $u,v \in \R$, consider
\begin{equation*}
    \begin{aligned}
    f_4(u,v)
    &= \pi\circ\exp (u\,X+v\,JX) \\
    &=\left[e^{-\frac{u}{3}}
    \begin{pmatrix}
     \frac{1}{6} e^{-v} \left(1+e^{2 v}+4 e^{u+v}\right) & -\frac{1}{\sqrt{3}} \sinh v& -\frac{e^{-v} \left(1+e^{2 v}-2 e^{u+v}\right)}{\sqrt{6}} \\
     -\frac{1}{\sqrt{3}}\sinh v & \cosh v & \sqrt{2} \sinh v \\
     -\frac{e^{-v}}{3 \sqrt{6}} \left(1+e^{2 v}-2 e^{u+v}\right) & \frac{\sqrt{2}}{3} \sinh v & \frac{1}{3} e^{-v} \left(1+e^{2 v}+e^{u+v}\right) \\
    \end{pmatrix}\right].
    \end{aligned}
\end{equation*}
The image of $f_4$ is a negative definite flat almost complex surface.
Moreover, it is the orbit through $[\id]$ of the Lie subgroup $\R^2\subset \Sl$ with Lie algebra $\mathrm{span}\{X,JX\}$.
As a consequence, this surface is isometric to $\R^2_2$.
Finally, since $\mathrm{span}\{X,JX\}$ is abelian, it follows from Lemma~\ref{lemma:sff} that the surface is totally geodesic.

\end{example}

\begin{example}
\label{ex:degenerateexample}
    For $u,v \in \R$, consider
    \[
    f_5(u,v) = 
    \pi \circ \exp\left(\tfrac{1}{\sqrt{2}} (u e_3+ v e_4) \right) =
    \left[ \begin{pmatrix}
        1 & 0 & u \\
        0 & 1 & v \\
        0 & 0 & 1
    \end{pmatrix}\right].
    \]
    The tangent space of the image of $f_5$ is $V_2$, and hence is preserved by $J$.
    As such, this surface is almost complex.
    Additionally, it is a degenerate surface and is the orbit through~$[\id]$ of the subgroup $\R^2\subset\Sl$.
    Finally, since the Lie algebra $\mathrm{span}\{e_3,e_4\}$ is included in $\mathfrak{m}$, it follows from Lemma~\ref{lemma:totgeocanembedded} that this surface is totally geodesic.
\end{example}

\subsection{Proof of Theorem \texorpdfstring{\ref{thm:main}}{A}}
Let $X$ be a unit tangent vector.
Note that, when the induced metric on $\Sigma$ is negative definite, we take $X$ with length $-1$.
After possibly restricting to an open dense subset, we may assume that locally either
\begin{enumerate}
\item[(1)]$X$ and therefore also $T\Sigma$, is contained in a single distribution $V_i$;
\item[(2)]$X$ and therefore also $T\Sigma$, is contained in the direct sum of the two
distributions $V_i\oplus V_j$;
\item[(3)]$X$ and therefore also $T\Sigma$, is contained in the direct sum of all three the
distributions.
\end{enumerate}
Note that the map $\varphi\colon \Sl/H\to\Sl/H:[A]\mapsto [(A^t)^{-1}]$ is an isometry not included in $\Sl$, the identity component of the isometry group of $\Sl/H$.
A direct computation shows that $\varphi_*$ preserves $V_1$ --meaning that $(V_1)_p$ is mapped into $(V_1)_{\varphi(p)}$ by $\varphi_*$-- and that $\varphi_*$ maps $(V_2)_p$ into $(V_3)_{\varphi(p)}$ and vice versa.
Hence, the cases in which $X$ is contained in $V_1\oplus V_2$ and $V_1\oplus V_3$ are related via an isometry of the ambient space.
Moreover, since $M$ is almost complex, $TM$ is spanned by $X$ and $JX$ and therefore we may take as basis elements any linear combination with length $\pm1$ between $X$ and $JX$.
Finally, we may also apply the adjoint action of $H$ to simplify $X$, without changing the surface $\Sigma$.

Taking the above arguments into account, we reduce the problem to the following cases, assuming the surface is non-degenerate:
\begin{enumerate}[(1)]
    \item $X=e_1$, \label{case:V1}
    \item $X=\varepsilon e_1+e_5$ with $\varepsilon\in\{-1,1\}$, \label{case:V1+V3}
    \item $X=\frac{1}{\sqrt{2}}(e_3+\varepsilon e_5)$ with $\varepsilon\in\{-1,1\}$, \label{case:V2+V3}
    \item $X=ae_1+e_3+\frac12 (a^2+\varepsilon )e_5+be_6$ with $a>0$ and $\varepsilon \in \{-1,1\}$. \label{case:V1+V2+V3}
\end{enumerate}
If we allow the surface to be degenerate, we obtain one additional case:
\begin{enumerate}[(5)]
    \item $X=e_3$. \label{case:V2}
\end{enumerate}

Note that the curvature tensor of the ambient space must preserve the tangent space of a totally geodesic surface, degenerate or non-degenerate.
From this fact, we deduce that Case~\ref{case:V1+V3} cannot occur.
Indeed, using the curvature~\eqref{eqn:curvature-tensor}, we obtain $\bar R(X,JX)JX = -4\varepsilon\, e_1 +2 e_5$, which is not contained in $\text{span}\{X,JX\}$.
Similarly,  in Case~\ref{case:V1+V2+V3} we deduce that $\varepsilon=-1$, $a = \frac{1}{\sqrt{3}}$ and $b=0$, by calculating $\bar R(X,JX)JX$.

Since $\Sigma$ is totally geodesic and almost complex by assumption, we know that $\Sigma$ must arise as the Riemannian exponential of $\text{span}\{X,JX\}$.
However, since the ambient space is naturally reductive, the Riemannian exponential at $o$ of a vector $Y \in \mathfrak m$ coincides with $\exp(Y)\cdot o$, where $\exp$ denotes the Lie group exponential~\cite[Thm.\ 5.1.2]{nikonorovbook}.
Hence, it follows immediately that Case~\ref{case:V1} corresponds to Example~\ref{ex:e1+e2}, Case~\ref{case:V2+V3} to Examples~\ref{ex:e3+e5} and~\ref{ex:e3-e5}, Case~\ref{case:V1+V2+V3} to Example~\ref{ex:e1+e3+e5} and Case~\ref{case:V2} to~Example~\ref{ex:degenerateexample}.

\bibliographystyle{abbrv}
\bibliography{references}

\end{document}